\documentclass[11pt,openbib]{amsart}

\usepackage{amsmath}
\usepackage{amssymb}
\usepackage{amsthm}
\usepackage{amscd}
\usepackage{mathrsfs}

\newcommand{\ri}{\mathfrak{o}}
\newcommand{\mi}{\mathfrak{p}}

\newcommand{\C}{\mathbf{C}}

\newcommand{\e}{\sqrt{\epsilon}}
\newcommand{\p}{\varpi}

\newcommand{\supp}{\operatorname{supp}}

\newcommand{\U}{\mathrm{U}}

\theoremstyle{plain}
\newtheorem{thm}[equation]{Theorem}
\newtheorem{lem}[equation]{Lemma}
\newtheorem{prop}[equation]{Proposition}
\newtheorem{cor}[equation]{Corollary}
\theoremstyle{definition}
\newtheorem{defn}[equation]{Definition}

\newtheorem{rem}[equation]{Remark}

\def\Section#1{\section{#1}\setcounter{equation}{0}}

\title{Conductors and newforms for non-supercuspidal representations of unramified $\U(2,1)$}
\author{Michitaka Miyauchi}
\date{}
\address{
Department of Mathematics, Faculty of Science\\
Kyoto University\\
Oiwake Kita-Shirakawa Sakyo Kyoto 606-8502 JAPAN
}
\email{miyauchi@math.kyoto-u.ac.jp}
\begin{document}

\begin{abstract} 
Let $G$ be the unramified unitary group in three variables
defined over a $p$-adic field with $p \neq 2$.
The conductors and newforms for representations of $G$
are defined by using
a certain family of open compact subgroups of $G$.
In this paper,
we determine the conductors of non-supercuspidal 
representations of $G$, and 
give explicit newforms in induced representations
for non-supercuspidal generic representations.
\end{abstract}

\maketitle
\pagestyle{myheadings}
\markboth{}{}

\section*{Introduction}
Local newforms play an important role
in the theory of automorphic representation.
We recall 
a result of Casselman \cite{Casselman}
for $\mathrm{GL}(2)$.
Let $F$ be a non-archimedean local field of characteristic zero
and $\ri_F$ its ring of integers 
with maximal ideal $\mi_F$.
For each non-negative integer $n$,
we define an open compact subgroup 
$\Gamma_0(\mi_F^n)$ of $\mathrm{GL}_2(F)$ by
\[
\Gamma_0(\mi_F^n) 
= 
\left(
\begin{array}{cc}
\ri_F & \ri_F \\
\mi_F^n & 1+\mi_F^n
\end{array}
\right)^\times.
\]
Let $(\pi, V)$ be an irreducible admissible
representation of $\mathrm{GL}_2(F)$.
We set
$V(n) = \{v \in V\, |\, \pi(k)v = v,\ k \in \Gamma_0(\mi_F^n)\}$.
Then the following holds:
\begin{thm}[\cite{Casselman}]\label{thm:gl2}
Suppose that an irreducible admissible representation
$(\pi, V)$ of $\mathrm{GL}_2(F)$
is generic.

(1) 
There exists a non-negative integer $n$ such that
$V(n)$ is not zero.

(2) Put $c(\pi) = \min \{n\geq 0\, |\, V(n) \neq \{0\}\}$.
Then the space $V(c(\pi))$ is one-dimensional.

(3)
For any integer $n \geq c(\pi)$,
\[
\dim V(n) = n-c(\pi)+1.
\]

(4)
The $\varepsilon$-factor of $\pi$ is a constant multiple
of $q_F^{-c(\pi)s}$
under suitable normalization,
where $q_F$ is the cardinality of the residue field of $F$.
\end{thm}

We call the integer $c(\pi)$ {\it the conductor of} $\pi$
and elements in $V(c(\pi))$ {\it newforms for} $\pi$.
In the proof of Theorem~\ref{thm:gl2},
Casselman showed implicitly
that newforms are test vectors for the 
appropriate Whittaker functional.
This property is important 
in the theory of zeta integrals.
After Casselman \cite{Casselman},
similar results are obtained by Jacquet,
Piatetski-Shapiro and Shalika \cite{JPSS} 
and Reeder \cite{Reeder} for $\mathrm{GL}_n(F)$
and by Roberts and Schmidt \cite{RS}
for $\mathrm{PGSp}(4)$.
For unitary groups,
there is a result by Lansky and Raguram \cite{LR}.
They computed the dimensions of the spaces of 
vectors fixed by certain open compact subgroups of 
unramified $\mathrm{U}(1,1)$.
However no comparison between conductors and exponents of
$\varepsilon$-factors has been studied for this group.

The author \cite{M} introduced the notion of newforms
for the unramified unitary group in three variables
defined over $F$.
We assume that $F$ is in addition of odd residual characteristic.
Let $E$ be the unramified quadratic  extension over $F$.
The group  unramified $\mathrm{U}(2,1)$
is realized as
$G = 
\{ g \in \mathrm{GL}_3(E)\ |\ 
{}^t\overline{g}Jg= J \}$,
where ${}^-$ is the non-trivial element in $\mathrm{Gal}(E/F)$
and
$J$ is $
\left(
\begin{array}{ccc}
0 & 0 &1\\
0 & 1 & 0\\
1 & 0 & 0
\end{array}
\right)$.
Newforms for $G$
is defined by a family of open compact subgroups
$\{K_n\}_{n \geq 0}$,
where
\begin{eqnarray*}
K_n
=
\left(
\begin{array}{ccc}
\ri_E & \ri_E &\mi^{-n}_E\\
\mi^n_E & 1+\mi^n_E & \ri_E\\
\mi^n_E & \mi^n_E & \ri_E
\end{array}
\right)
\cap G.
\end{eqnarray*}
For a smooth representation $(\pi, V)$ of $G$,
we denote by
$V(n)$ the space of $K_n$-fixed vectors in $V$.
We say that a smooth representation $(\pi, V)$
admits a newform
if $V(n)$ is not zero for some $n \geq 0$.
In \cite{M},
the author showed that
every irreducible generic representation of $G$
admits a newform.
The integer $N_\pi = \min \{n \geq 0\, |\, V(n) \neq \{0\}\}$
is called the conductor of $\pi$.
We call $V(N_\pi)$ the space of newforms for $\pi$
and $V(n)$ that of oldforms when $n > N_\pi$.
The main theorem of \cite{M}
is the multiplicity one property of newforms,
that is, 
for each irreducible admissible representation $\pi$ of $G$
admitting a newform,
the space of newforms for $\pi$ is one-dimensional.

One of aims of this paper 
is to show that
newforms for $G$ 
are  test vectors for the appropriate Whittaker functional.
This property is important for the application
to the theory of zeta integral.
In \cite{M2},
the author applied newform theory  for $G$
to Rankin-Selberg type zeta integrals of Gelbart,
Piatetski-Shapiro \cite{GPS} and Baruch \cite{Baruch},
and proved that 
zeta integrals of newforms for generic supercuspidal representations
agree with their $L$-factors.
This property means that 
zeta integrals of newforms do not  vanish.

To show that
newforms are test vectors for the Whittaker functional,
we determine the newforms for each non-supercuspidal
generic representation explicitly.
This property was already proved in \cite{M}
for a certain class of representations of $G$, which contains the generic supercuspidal
representations.
Therefore  we need to consider only non-supercuspidal
representations.
Similar results are obtained
by Lansky and Raguram for
unramified $\mathrm{U}(1,1)$ 
and for $\mathrm{SL}(2)$(\cite{LR} and \cite{LRS}),
and by Roberts and Schmidt for $\mathrm{PGSp}(4)$ (\cite{RS}).

Firstly,
we determine conductors and oldforms for
the parabolically induced representations from the 
Borel subgroup $B$ of $G$.
A parabolically induced representation $(\pi, V)$
of $G$ is induced from a quasi-character $\mu_1$
of $E^\times$
and a character $\mu_2$ of the norm-one subgroup $E^1$ of 
$E^\times$.
By the theory of Bernstein and Zelevinsky,
to obtain a basis for $V(n)$,
we only need to determine the elements
$g$ in $B\backslash G/K_n$
such that $\mu_1 \otimes \mu_2$ is trivial on $B\cap g K_n g^{-1}$.
It will turn out that the conductor of $\mathrm{Ind}_B^G \mu_1 \otimes \mu_2$
equals to $2c(\mu_1) + c(\mu_2)$,
where $c(\mu_i)$ is the conductor of $\mu_i$ (see Theorem~\ref{thm:cond}).

Secondly,
we determine the conductors of
irreducible subquotients
of reducible parabolically induced representations
according to the classification by Keys \cite{Keys}.
Here 
the level raising operator $\theta'$
on the space of newforms
plays an important role.
If we find a vector which does not vanish under $\theta'$,
then this vector should belong to 
the generic constituent of the parabolically induced representation.
The method in this part is rather technical.

Thirdly,
we determine newforms for generic non-supercuspidal
representations explicitly.
Every irreducible 
non-supercuspidal representation $\pi$
of $G$ can be embedded into
$\mathrm{Ind}_B^G \mu$,
for some quasi-character $\mu$ of 
the diagonal torus $T$.
We realize the newform for $\pi$ 
as a function in $\mathrm{Ind}_B^G \mu$ explicitly.
This is easy if $\pi$ is not the Steinberg representation.
The space of newform for the Steinberg representation
is characterized as the kernel of the level lowering operator
on oldforms for $\mathrm{Ind}_B^G \mu$.

Finally,
we prove that newforms 
are test vectors for the appropriate Whittaker functional.
Here the level raising operator $\theta'$
on the space of newforms
plays an important role again.
Using the explicit newforms,
we show that the injectivity of $\theta'$.
Then this property follows automatically (see Lemma~\ref{lem:theta}).
As a corollary,
we obtain the following dimension formula of oldforms.
Similar to the case of
$\mathrm{GL}(n)$
and $\mathrm{GSp}(4)$,
the growth of dimensions of oldforms for generic representations
$\pi$
of $G$ is independent of the choice of $\pi$.
\begin{thm}
Let $(\pi, V)$ be an irreducible generic representation of $G$.
For $n \geq N_\pi$, we have
\[
\dim V(n) = \left\lfloor \frac{n-N_\pi}{2}\right\rfloor +1.
\]
\end{thm}

We hope that 
results in this paper
are useful for
the theory of zeta integral,
especially
Rankin-Selberg type one of Gelbart and 
Piatetski-Shapiro.

\medskip
\noindent
{\bf Acknowledgements} \
The author would like to thank 
Takuya Yamauchi 
for useful discussions,
and Yoshi-hiro Ishikawa for comments
on a previous version of this paper.

\Section{Preliminaries}
In subsection~\ref{subsec:notation},
we fix notation for the unramified unitary group in three variables,
which is used in this paper.
In subsection~\ref{subsec:newform},
we recall from \cite{M} the definition and some properties of
newforms for unramified $\U(2,1)$.
As in \cite{M},
the level raising operator $\theta'$ on the space of newforms
plays an important role.
We will prepare Lemma~\ref{lem:theta}, which 
is a main tool
in our investigation.
\subsection{Notation}\label{subsec:notation}
Let $F$ be a non-archimedean local field of characteristic zero,
$\ri_F$ its ring of integers,
$\mi_F = \p_F \ri_F$ the maximal ideal in $\ri_F$ and
$k_F = \ri_F/\mi_F$ the residue field of $F$.
We write $q = q_F$ for the cardinality of $k_F$.
Let $|\cdot|_F$ be the absolute value of $F$
normalized so that $|\p_F|_F = q_F^{-1}$.
We use the analogous notation 
for any non-archimedean local field.
Throughout this paper,
we  assume that
$F$ is of odd
residual characteristic.

Let $E = F[\e]$ be the quadratic unramified   extension over $F$,
where  $\epsilon$ is a non-square element in $\ri_F^\times$.
Then $\p_F$ is a uniformizer of $E$
and we abbreviate $\p = \p_F$.
We realize the $F$-points of the unramified unitary group
in three variables defined over $F$ as
$G = 
\{ g \in \mathrm{GL}_3(E)\ |\ 
{}^t\overline{g} Jg = J \}$,
where 
${}^-$ is the non-trivial element in $\mathrm{Gal}(E/F)$
and
\begin{eqnarray*}
J = 
\left(
\begin{array}{ccc}
0 & 0 &1\\
0 & 1 & 0\\
1 & 0 & 0
\end{array}
\right).
\end{eqnarray*}

Let $B$ be the Borel subgroup of $G$ consisting of 
the upper triangular elements in $G$,
$T$ the Levi subgroup of $B$
and $U$ the unipotent radical of $B$.
We write $\hat{U}$ for the opposite of $U$.
Then we have
\begin{eqnarray*}
U = \left\{ u(x,y) =
\left(
\begin{array}{ccc}
1 & x & y\\
0 & 1 & -\overline{x}\\
0 & 0 & 1
\end{array}
\right)\, 
\Bigg|\,
x, y \in E,\
y+\overline{y} +x\overline{x} = 0
\right\}
\end{eqnarray*}
and
\begin{eqnarray*}
\hat{U} = \left\{ \hat{u}(x,y) =
\left(
\begin{array}{ccc}
1 & 0 & 0\\
x & 1 & 0\\
y & -\overline{x} & 1
\end{array}
\right)\, 
\Bigg|\,
x, y \in E,\
y+\overline{y} +x\overline{x} = 0
\right\}.
\end{eqnarray*}

We fix a non-trivial additive character $\psi_E$ of $E$
with conductor $\ri_E$
and define a character $\psi$ of $U$ by
\[
\psi(u(x, y)) = \psi_E(x),\ \mathrm{for}\
u(x,y) \in U.
\]
We say that
a smooth representation $\pi$ of $G$ is {\it generic}
if $\mathrm{Hom}_U(\pi, \psi) \neq \{0\}$.
For an
irreducible admissible representation $\pi$
of $G$,
it is well-known that 
\[
\dim \mathrm{Hom}_U(\pi, \psi) \leq 1.
\]
If $(\pi, V)$ is an irreducible generic representation of $G$,
then
by Frobenius reciprocity, we have
\[
\mathrm{Hom}_G(\pi, \mathrm{Ind}_U^G \psi)
\simeq
\mathrm{Hom}_U(\pi, \psi) 
\simeq \C.
\]
So there exists a unique embedding of
$\pi$ into $\mathrm{Ind}_U^G \psi$
up to scalar.
The image $\mathcal{W}(\pi, \psi)$ of $V$ is called {\it the Whittaker model of} $\pi$.
By a non-zero functional $l \in \mathrm{Hom}_U(\pi, \psi)$,
which 
is called {\it the Whittaker functional},
we define {\it the Whittaker function} $W_v \in \mathcal{W}(\pi, \psi)$
associated to $v \in V$ by
\[
W_v(g) = l(\pi(g)v),\ g \in G.
\]

We set
\begin{eqnarray*}
& T_H = \left\{
t(a) = \left(
\begin{array}{ccc}
a & 0 & 0\\
0 & 1 & 0\\
0 & 0 & \overline{a}^{-1}
\end{array}
\right)\, \Bigg|\,
a \in E^\times
\right\}.
\end{eqnarray*}
Then the group $T_H$ is isomorphic to $E^\times$.
Let $(\pi, V)$ be an irreducible generic representation of $G$.
For each $v \in V$,
we can regard the restriction $W_v|_{T_H}$ of $W_v$ to $T_H$
as a locally constant function on $E^\times$.
Along the lines of the Kirillov theory for $\mathrm{GL}(2)$,
we see that there exists an integer $n$ such that
$\supp W_v|_{T_H} \subset \mi_E^{-n}$.
Moreover, if $v$ lies in $\langle \pi(u)w-w\ |\ u \in U,\ w \in V \rangle$,
then $W_v|_{T_H}$ is a compactly supported function on $E^\times$.

The conductor of a quasi-character $\mu_1$
of $E^\times$ is defined by
\[
c(\mu_1) = \mathrm{min}\{n \geq 0\ |\ \mu_1|_{(1+\mi_E^{n})\cap \ri_E^\times} = 1\}.
\]
We say that $\mu_1$ is {\it unramified}
if $c(\mu_1) = 0$.
We set open compact subgroups
of the norm-one subgroup $E^1$ of $E^\times$
as
\begin{eqnarray*}
& 
E^{1}_0 = E^{1},\
E^{1}_n = E^{1}\cap (1+\mi_E^n),\ \mathrm{for}\ n \geq 1.
\end{eqnarray*}
We define the conductor
of a character $\mu_2$ of $E^1$ by
\[
c(\mu_2) = \mathrm{min}\{n \geq 0\ |\ \mu_2|_{E^1_n} = 1\}.
\]
There exists an isomorphism between $E^1$ and
the center $Z$ of $G$ given by
\[
\iota: E^1 \simeq Z;
\lambda \mapsto
\left(
\begin{array}{ccc}
\lambda & & \\
& \lambda & \\
& & \lambda
\end{array}
\right).
\]
If a smooth representation $\pi$ of $G$
admits the central character $\omega_\pi$, then
we define 
{\it the conductor of} $\omega_\pi$ by
\[
n_\pi =  \mathrm{min}\{n \geq 0\ |\ \omega_\pi|_{Z_n} = 1\},
\]
where $Z_n = \iota(E^1_n)$ for $n \geq 0$.

We shall prepare the following lemma on the structure of $E^1$.
\begin{lem}\label{lem:E^1}
Suppose that a subgroup $H$ of $E^1$
contains the set $\{ (1-a\e)(1+a\e)^{-1}\ |\ a \in \ri_F\}$.
Then we have $H = E^1$.
\end{lem}
\begin{proof}
By \cite{Morris-2} Theorem 2.13 (c),
the group $E^1_1$ coincides with 
$\{ (1-a\e)(1+a\e)^{-1}\ |\ a \in \mi_F\}$.
So $H$ contains $E^1_1$ by assumption.
The quotient $E^1/E^1_1$
is isomorphic to the norm-one subgroup $k_E^1$
of $k_E^\times$,
and hence it
is a cyclic group of order $q+1$.
We claim that for $a, b \in \ri_F$,
$(1-a\e)(1+a\e)^{-1} \equiv (1-b\e)(1+b\e)^{-1} \pmod{E^1_1}$
implies
$a \equiv b \pmod{\mi_F}$.
Then
the group $H/E^1_1$ contains at least $q$-elements.
Thus we get $H/E^1_1 = E^1/E^1_1$,
whence $H = E^1$.

We shall prove the claim.
Suppose that two elements
$a, b$ in $\ri_F$ satisfy
$(1-a\e)(1+a\e)^{-1} \equiv (1-b\e)(1+b\e)^{-1} \pmod{E^1_1}$.
Then we have
$(1-a\e)(1+a\e)^{-1} - (1-b\e)(1+b\e)^{-1} \in \mi_E$.
Since 
$(1+a\e) (1+b\e)$ lies in $\ri_E^\times$,
we obtain
$(1-a\e)(1+b\e)-(1-b\e)(1+a\e) 
= 2(b-a)\e \in \mi_E$.
This means $b-a \in \mi_F$
since we are assuming that $F$ is of odd residual characteristic.
\end{proof}

\subsection{Newforms}\label{subsec:newform}
For each non-negative integer $n$,
we define an open compact subgroup $K_n$ of $G$
by
\begin{eqnarray*}
K_n
=
\left(
\begin{array}{ccc}
\ri_E & \ri_E &\mi^{-n}_E\\
\mi^n_E & 1+\mi^n_E & \ri_E\\
\mi^n_E & \mi^n_E & \ri_E
\end{array}
\right)
\cap G.
\end{eqnarray*}
We put
\begin{eqnarray*}
t_{n}
= \left(
\begin{array}{ccc}
 & & \p^{-n}\\
 & 1 & \\
 \p^{n} & &
\end{array}
\right) \in K_n.
\end{eqnarray*}
For a smooth representation $(\pi, V)$ of $G$,
we set
\begin{eqnarray*}
V(n) = \{v \in V\ |\ \pi(k)v= v,\ k \in K_n \},\ n \geq 0.
\end{eqnarray*}
\begin{defn}[\cite{M} Definition 2.6]
Suppose that a smooth representation $(\pi, V)$ of $G$
has a non-zero $K_n$-fixed vector, for some $n \geq 0$.
We define {\it the conductor of} $\pi$ by
$N_\pi = \mathrm{min}\{n\geq 0\, |\, V(n) \neq \{0\}\}$.
We call $V(N_\pi)$ {\it the space of newforms for} $\pi$
and $V(n)$ {\it that of oldforms},
for $n > N_\pi$.
\end{defn}
We say that a smooth representation $\pi$ admits a newform
if $\pi$ has a non-zero $K_n$-fixed vector, for some $n \geq 0$.
\begin{thm}[\cite{M} Theorems 2.8, 5.6]\label{thm:new}
(i) Every irreducible generic representation of $G$
admits a newform.

(ii) 
If an  irreducible admissible representation $(\pi, V)$
of $G$ admits a newform,
then $V(N_\pi)$ is one-dimensional.
\end{thm}
\begin{rem}
Suppose that a smooth representation $\pi$ of $G$ 
admits the central character $\omega_\pi$.
If $\pi$ has a non-zero $K_n$-fixed vector,
then $\omega_\pi$ is trivial on $Z_n = Z\cap K_n$.
So we get
\begin{eqnarray*}
N_\pi \geq n_\pi.
\end{eqnarray*}
\end{rem}

In \cite{M},
two level raising operators, which are inspired by those in \cite{RS},
played an important role
to investigate $K_n$-fixed vectors.
The first one
$\eta: V(n) \rightarrow V(n+2)$ 
is given by
\begin{eqnarray*}
\eta v = \pi(\zeta^{-1}) v,\ v \in V(n),
\end{eqnarray*}
where
\begin{eqnarray*}
\zeta = 
\left(
\begin{array}{ccc}
\p & & \\
& 1 & \\
& & \p^{-1}
\end{array}
\right).
\end{eqnarray*}
The second level raising operator $\theta': V(n) \rightarrow V(n+1)$
is defined by
\begin{eqnarray*}
\theta' v 
= \frac{1}{\mathrm{vol}(K_{n+1}\cap K_n)}
\int_{K_{n+1}} \pi(k) v dk,\ v \in V(n).
\end{eqnarray*}
By \cite{M} Proposition 3.3,
we have
\begin{eqnarray}\label{eq:theta}
\theta' v
= \eta v +
\sum_{x \in \mi_F^{-1-n}/\mi_F^{-n}}
\pi\left(
\begin{array}{ccc}
1 & & x\e\\
 & 1 & \\
  &  & 1
\end{array}
\right)v,\ v \in V(n).
\end{eqnarray}

If newforms for $\pi$ are  test vectors for the Whittaker functional,
we have the following dimension formula of oldforms:
\begin{prop}[\cite{M} Theorem 5.8]\label{prop:oldform}
Let $(\pi, V)$ be an irreducible generic representation of $G$.
Suppose that $W_v(1) \neq 0$ for all
non-zero elements  $v$ in $V(N_\pi)$.
Then, for $n \geq N_\pi$,
the set
$\{\theta^{'i} \eta^j v\ |\ i+2j+N_\pi = n\}$
constitutes a basis for $V(n)$.
In particular,
\[
\dim V(n) = \left\lfloor \frac{n-N_\pi}{2}\right\rfloor +1.
\]
\end{prop}

It follows from \cite{M} Theorem 4.12 
that the assumption of Proposition~\ref{prop:oldform}
holds for irreducible generic representations of $G$ which satisfy
$N_\pi \geq 2$ and $N_\pi > n_\pi$.
One of the aims of this paper
is to prove the above dimension formula of oldforms
for all generic representations of $G$.
The following lemma gives a criterion
for newforms to be 
test vectors for the Whittaker functional.
\begin{lem}\label{lem:theta}
Let $(\pi, V)$ be an irreducible admissible representation of $G$ which admits a newform.
Suppose that $N_\pi \geq 1$
and the level raising operator $\theta': V(N_\pi) \rightarrow
V(N_\pi+1)$ is injective.
Then 

(i) $\pi$ is generic;

(ii) For all non-zero elements $v$ in $V(N_\pi)$,
we have
$W_v(1) \neq 0$.
\end{lem}
\begin{proof}
(i)
By assumption,
the space $V(N_\pi+1)$ is not zero.
Thus it follows from \cite{M} Theorem 5.10 (ii)
that $\pi$ must be generic.

(ii)
Due to \cite{M} Corollary 4.6,
for $v \in V(n)$,
the function $W_v|_{T_H}$ is $\ri_E^\times$-invariant
and its support is contained in  $\ri_E$.
Let $v$ be an element in $V(N_\pi)$
such that $W_v(1) = 0$.
By (\ref{eq:theta}),
we have
\[
W_{\theta' v}(t(a)) = W_v(t(a\p^{-1}))+ qW_v(t(a)),
\]
for $a \in E^\times$,
so that
 $W_{\theta' v}(1) = W_v(t(\p^{-1}))+qW_v(1) = 0$.
This implies that
the support of  $W_{\theta' v}|_{T_H}$ is  contained in $\mi_E$.
It follows from \cite{M} Lemma 4.9
that
$\theta' v$ lies in $\eta V(N_\pi-1) = \{0\}$.
So we get $v = 0$
by the injectivity of $\theta'$.
Now the proof is complete.
\end{proof}

\Section{Conductors of parabolically induced representations}
\label{sec:full}
In this section,
we determine the conductors and oldforms for
the parabolically induced representations of $G$.
Here we do not assume that they are irreducible.
We use the following notation for parabolically induced
representations of $G$.
Given a quasi-character $\mu_1$ of $E^\times$
and a character $\mu_2$ of $E^{1}$,
we define a quasi-character $\mu = \mu_1 \otimes \mu_2$ of $T$ by
\begin{eqnarray*}
\mu\left(
\begin{array}{ccc}
a & & \\
& b & \\
& & \overline{a}^{-1}
\end{array}
\right)
= \mu_1(a)\mu_2(b),\ \mathrm{for}\ a \in E^\times\ \mathrm{and}\ b \in E^{1}.
\end{eqnarray*}
We regard $\mu$ as a quasi-character of $B$ 
by extending
trivially on $U$.
Let $\pi = \mathrm{Ind}_B^G(\mu)$ be
the normalized parabolic induction.
Then the space $V$ of $\pi$
is that of locally constant functions $f: G \rightarrow \C$
which satisfy
\[
f(bg) = \delta_B(b)^{1/2}\mu(b)f(g),\ \mathrm{for}\
b \in B,\ g \in G,
\]
where $\delta_B$ is the modulus character of $B$.
We note that 
\begin{eqnarray*}
\delta_B\left(
\begin{array}{ccc}
a & & \\
& b & \\
& & \overline{a}^{-1}
\end{array}
\right)
= |a|_E^2,\ \mathrm{for}\ a \in E^\times,\ b \in E^{1}.
\end{eqnarray*}

For a non-negative integer $n$,
it follows from \cite{BZ} (2.24) that
 $V(n)$ 
is spanned by the functions
supported on $BgK_n$,
where $g$ runs over the elements in $B\backslash G/K_n$
such that $\mu$ is trivial on $B\cap g K_n g^{-1}$.
For any set $S$,
we denote by $\mathrm{Card}\, S$ the cardinality of $S$.
Thus we have
\[
\dim V(n) = \mathrm{Card}\{g \in B\backslash G/K_n\ |\
\mu|_{B\cap g K_n g^{-1}} = 1\}.
\]

We shall give
 a complete set of representatives for 
$B\backslash G/K_n$.
For any integer $i$,
we set
\begin{eqnarray*}
\gamma_i
=
\left(
\begin{array}{ccc}
1 & & \\
\p^i & 1 & \\
-\p^{2i}/2 & -\p^i & 1
\end{array}
\right) \in G.
\end{eqnarray*}
\begin{lem}\label{lem:coset}
For a non-negative integer $n$,
 a complete set of representatives for 
$B\backslash G/K_n$
is given by $\lfloor \frac{n}{2}\rfloor +1$ elements
\[
\gamma_i,\ \left\lceil \frac{n}{2}\right\rceil \leq i \leq n.
\]
\end{lem}
\begin{proof}
Firstly, we claim that 
$\mathrm{Card} (B\backslash G/K_n) = \lfloor \frac{n}{2}\rfloor +1$.
It follows \cite{Keys} that
there exists 
an unramified quasi-character $\mu_1$ of $E^\times$
such that
$\pi = \mathrm{Ind}_B^G \mu_1 \otimes 1$
is irreducible.
Clearly, $\mu_1 \otimes 1$ is trivial 
on $B\cap g K_n g^{-1}$
for all $g \in G$.
This means
$\mathrm{Card} (B\backslash G/K_n) =
\dim V(n)$.
The representation $\pi = \mathrm{Ind}_B^G \mu_1 \otimes 1$ is generic and unramified,
that is, $\pi$ has conductor zero.
Then \cite{CS} Theorem 5.4 says that
$W_v(1) \neq 0$
for every non-zero element $v$ in $V(0)$.
So by Proposition~\ref{prop:oldform}, we have
\[
\mathrm{Card} (B\backslash G/K_n) =
\dim V(n) = \left\lfloor \frac{n}{2}\right\rfloor +1,
\]
as required.

Secondly, we show that
every element in $G$ lies in $B\gamma_i K_n$,
for some 
$\left\lceil \frac{n}{2}\right\rceil \leq i \leq n$.
If $n =0$,
the assertion follows from the Iwasawa decomposition
$G = BK_0$.
The group $K_0/(1+M_3(\mi_E))\cap G$ is isomorphic to 
$\U(2,1)(k_E/k_F)$.
Using the Bruhat decomposition
of $\U(2,1)(k_E/k_F)$,
we get 
$K_0 = (B\cap K_0) W (K_0\cap K_1)$,
where 
$W = \{1, t_0\}$.
Thus we have
$G = B K_0 = BW (K_0\cap K_1) = BK_1$
since $t_0 \in Bt_1 \subset BK_1$.
This completes the proof for $n = 1$.
Suppose that $n \geq 2$.
Set $I = Z(K_0\cap K_1)$.
Then $I$ is the standard Iwahori subgroup of $G$
and have the Iwahori decomposition $I = 
(I\cap \hat{U})(I\cap T)(I\cap U)$.
So we get 
$G = BWI = B(I\cap \hat{U}) \cup Bt_0 (I\cap U)$,
and hence
$G = G\zeta^j =  B(I\cap \hat{U})\zeta^j\cup Bt_0 (I\cap U)\zeta^j$,
where $j = \lfloor \frac{n}{2}\rfloor$.
Set
\[
\hat{U}(\mi_E^{\lceil \frac{n}{2}\rceil })
= 
\left(
\begin{array}{ccc}
1 & & \\
\mi_E^{\lceil \frac{n}{2}\rceil }& 1 & \\
\mi_E^n&\mi_E^{\lceil \frac{n}{2}\rceil } & 1
\end{array}
\right)\cap G.
\]
We shall claim that
$G = B \hat{U}(\mi_E^{\lceil \frac{n}{2}\rceil })K_n$.
Clearly,
$B(I\cap \hat{U})\zeta^j =B\zeta^{-j}(I\cap \hat{U})\zeta^j$
is contained in $B \hat{U}(\mi_E^{\lceil \frac{n}{2}\rceil })K_n$.
Moreover, we have
$Bt_0 (I\cap U)\zeta^j \subset 
Bt_0 (I\cap U)\zeta^jK_n = 
B\zeta^{j-n}t_0 (I\cap U)\zeta^j t_n K_n
= Bt_{n-j} (I\cap U) t_{n-j} K_n \subset B \hat{U}(\mi_E^{\lceil \frac{n}{2}\rceil })K_n$.
So we obtain
$G = B \hat{U}(\mi_E^{\lceil \frac{n}{2}\rceil })K_n$,
as required.
Every element in $\hat{U}(\mi_E^{\lceil \frac{n}{2}\rceil })$
can be written as
$\hat{u}(y, a\e -y\overline{y}/2)$,
where $y \in \mi_E^{\lceil \frac{n}{2}\rceil}$
and $a \in \mi_F^n$.
Since
$\hat{u}(y, a\e -y\overline{y}/2)
= \hat{u}(y, -y\overline{y}/2)\hat{u}(0, a\e)$
and $\hat{u}(0, a\e) \in K_n$,
any element in $G$
belongs to $B \hat{u}(y,  -y\overline{y}/2) K_n$,
for some $y \in \mi_E^{\lceil \frac{n}{2}\rceil}$.
If $y$ lies in $\mi_E^n$, then the element
$\hat{u}(y,  -y\overline{y}/2)$ belongs to $K_n$.
Hence we have
$B \hat{u}(y,  -y\overline{y}/2) K_n = BK_n = B\gamma_n K_n$.
Suppose that $y$ lies in 
$\mi_E^{\lceil \frac{n}{2}\rceil}\backslash
\mi_E^n$.
Put $i = \nu_E(y)$,
where $\nu_E$ is the valuation on $E$ normalized so that
$\nu_E(\p) = 1$.
Then we have
$B \hat{u}(y, -y\overline{y}/2) K_n
= B\gamma_i K_n$
since $B$ and $K_n$ contain
$T_H\cap K_0$.
This completes the proof.
\end{proof}

The following lemma determines the condition on $i$
such that $\mu$ is trivial on $B\cap \gamma_i K_n \gamma_i^{-1}$.
\begin{lem}\label{lem:intertwine}
Let $i$ be an integer such that
$\left\lceil \frac{n}{2}\right\rceil \leq i \leq n$.
Then 
$\mu$ is trivial on 
$B\cap \gamma_i K_n \gamma_i^{-1}$
if and only if
\[
c(\mu_2) \leq 2i-n\ and\
c(\mu_1) \leq n-i.
\]
\end{lem}
\begin{proof}
We can write an element $g$ in $B = TU$
as
\[
g = \left(
\begin{array}{ccc}
a & & \\
& b&\\
& & \overline{a}^{-1}
\end{array}
\right)
\left(
\begin{array}{ccc}
1 & x & y\\
& 1 & -\overline{x}\\
& & 1
\end{array}
\right)
=
\left(
\begin{array}{ccc}
a & ax &a y\\
& b & -b\overline{x}\\
& & \overline{a}^{-1}
\end{array}
\right),\
\]
where 
$a \in E^\times$, $b \in E^{1}$
and
$x, y \in E$ such that $y+\overline{y}+x\overline{x} = 0$.
It is easy to observe that
$a \in (1+\mi_E^{n-i})\cap \ri_E^\times$ and $b \in E^1_{2i-n}$
if $g \in \gamma_i K_n \gamma_i^{-1}$.
This implies that
$\mu$ is trivial on 
$B\cap \gamma_i K_n \gamma_i^{-1}$
if 
$c(\mu_2) \leq 2i-n$ and
$c(\mu_1) \leq n-i$.

Suppose that 
$\mu$ is trivial on 
$B\cap \gamma_i K_n \gamma_i^{-1}$.
For all $a \in (1+\mi_E^{n-i})\cap \ri_E^\times$,
the element $t(a)$ belongs to 
$B\cap \gamma_i K_n \gamma_i^{-1}$.
So we get
$c(\mu_1) \leq n-i$.
We  claim that
for each $b$ in $E^1_{2i-n}$,
there exists an element $g$ in $B\cap \gamma_i K_n \gamma_i^{-1}$ whose $(2,2)$-entry is $b$.
Then we get $\mu_2(b) =\mu(g) = 1$
because we have seen that 
the $(1,1)$-entry of $g \in B\cap \gamma_i K_n \gamma_i^{-1}$ lies in $(1+\mi_E^{n-i})\cap \ri_E^\times$
and $c(\mu_1) \leq n-i$.
So we obtain $c(\mu_2) \leq 2i-n$,
which completes the proof.

We shall show the claim.
Set 
\[
\mathcal{K}
=
\left(
\begin{array}{ccc}
\ri_E & \ri_E &\mi^{-n}_E\\
\mi^n_E & \mi^n_E & \ri_E\\
\mi^n_E & \mi^n_E & \ri_E
\end{array}
\right).
\]
We define an involution $\sigma$ on $M_3(E)$
by $\sigma(X) = J {}^t \overline{X}J$, for $X \in M_3(E)$.
Then we have $G = \{g \in \mathrm{GL}_3(E)\ |\ g \sigma(g) = 1\}$
and $K_n = (1+\mathcal{K})\cap G$.
For an element $a$ in $\mi_F^{2i-n}$,
we set
\[
X
=
\left(
\begin{array}{ccc}
0 & \p^{-i}a\e & \p^{-2i}a\e\\
0 & a\e & \p^{-i}a\e\\
0 & 0 & 0
\end{array}
\right).
\]
Then one can observe that 
$1-X$ lies in $\gamma_i (1 + \mathcal{K})^\times \gamma_i^{-1}$.
Since the group $\gamma_i (1 + \mathcal{K})^\times \gamma_i^{-1}$
is $\sigma$-stable,
the element $1+X = \sigma(1-X)$ belongs to 
$\gamma_i (1 + \mathcal{K})^\times \gamma_i^{-1}$.
We put $g = (1-X)(1+X)^{-1}$.
Then we see that
$g$ is an element in $B\cap \gamma_i K_n \gamma_i^{-1}$
whose $(2,2)$-entry is $(1-a\e)(1+a\e)^{-1}$.
Suppose that $2i-n > 0$.
Then \cite{Morris-2} Theorem 2.13 (c)
says that
every element in $E^1_{2i-n}$ has
the form $(1-a\e)(1+a\e)^{-1}$, for some $a \in \mi_F^{2i-n}$.
This completes the proof of the claim for $2i -n >0$.
Suppose that $2i-n = 0$.
Then we have showed that
the subgroup of $E^1$ consisting of the $(2,2)$-entries of elements in 
$B\cap \gamma_i K_n \gamma_i^{-1}$
contains $\{(1-a\e)(1+a\e)^{-1}\ |\ a \in \ri_F\}$.
Now the claim follows from Lemma~\ref{lem:E^1}.
\end{proof}

Here
we introduce functions in $\pi = \mathrm{Ind}_B^G \mu$,
which form a basis for $V(n)$.
\begin{defn}\label{defn:f}
For $\left\lceil \frac{n}{2}\right\rceil \leq i \leq n$,
we denote by $f_{n, i}$ the function in $V(n)$
which satisfies
$f_{n, i}(\gamma_{i}) = 1$ and $\supp f_{n, i} = B\gamma_i K_n$.
By Lemma~\ref{lem:intertwine}, $f_{n, i}$
is well-defined 
if and only if
$i$ satisfies
$c(\mu_2) \leq 2i-n$ and
$c(\mu_1) \leq n-i$.
\end{defn}

We shall determine the conductors
of the parabolically induced representations of $G$,
and give the dimension formula for 
the spaces of their oldforms.
\begin{thm}\label{thm:cond}
Let $\pi$ be a parabolically induced representation
$\mathrm{Ind}_B^G(\mu_1\otimes \mu_2)$
of $G$,
where 
$\mu_1$ is a quasi-character of $E^\times$
and $\mu_2$ is a character of $E^{1}$.
Then

(i)
$N_\pi = 2c(\mu_1) +c(\mu_2)$;

(ii) For $n \geq N_\pi$,
the functions
$f_{n, i}$, $\frac{n+c(\mu_2)}{2}\leq i \leq n-c(\mu_1)$ constitute a basis for $V(n)$.
In particular,
\[
\dim V(n) = \left\lfloor \frac{n-N_\pi}{2}\right\rfloor +
1.
\]
\end{thm}

\begin{proof}
The theorem follows from 
Lemmas~\ref{lem:coset} and \ref{lem:intertwine}.
\end{proof}

The representation $\pi = \mathrm{Ind}_B^G(\mu_1\otimes \mu_2)$ admits the central character $\omega_\pi$,
which is given by
\begin{eqnarray}\label{eq:cent_f}
\omega_\pi(\iota(b)) = \mu_1(b) \mu_2(b),\
b \in E^{1}.
\end{eqnarray}

\begin{prop}\label{prop:cent}
Let $\pi = \mathrm{Ind}_B^G(\mu_1 \otimes \mu_2)$
be a parabolically induced representation of $G$.
Then $N_\pi = n_\pi$
if and only if $c(\mu_1) = 0$.
If this is the case,
then we have $N_\pi = n_\pi = c(\mu_2)$.
\end{prop}
\begin{proof}
By (\ref{eq:cent_f}),
we have
$n_\pi \leq \max \{c(\mu_1), c(\mu_2)\}$.
Suppose that $N_\pi = n_\pi$.
Then by Theorem~\ref{thm:cond} (i),
we obtain
$\max \{c(\mu_1), c(\mu_2)\}
= 2c(\mu_1)+c(\mu_2)$,
whence $c(\mu_1) = 0$.
Conversely, suppose that $c(\mu_1) = 0$.
Then we have $\omega_\pi(\iota(b)) = \mu_2(b)$,
for $b \in E^1$.
So we get
$n_\pi = c(\mu_2) = N_\pi$
by Theorem~\ref{thm:cond} (i).
\end{proof}

\Section{Reducible case}\label{sec:red}
Every irreducible non-supercuspidal representation 
of $G$ is a subquotient of 
$\mathrm{Ind}_B^G \mu$,
for some quasi-character $\mu$ of $T$.
In this section,
we determine the conductors of 
the irreducible subquotients of 
$\mathrm{Ind}_B^G \mu$
in the case when $\mathrm{Ind}_B^G \mu$
is reducible.

\subsection{Reducibility of parabolically induced representations}
Suppose that a parabolically induced representation 
$\pi = \mathrm{Ind}_B^G \mu_1 \otimes \mu_2$ is reducible.
Due to \cite{Keys},
the reducibility of $\pi$
is determined by
the quasi-character $\widetilde{\mu}_1$ of $E^\times$ 
which is
given by
\[
\widetilde{\mu}_1(a) = \mu_1(a)\mu_2(\overline{a}/a),\
a \in E^\times.
\]
There exist the following three reducible cases:
\begin{enumerate}
\item[(R1)]
$\widetilde{\mu}_1 = |\cdot|_E^{\pm}$;

\item[(R2)]
$\widetilde{\mu}_1|_{F^\times} = \omega_{E/F} |\cdot|_F^{\pm}$,
where $\omega_{E/F}$ is the non-trivial character of $F^\times$
which is trivial on $N_{E/F}(E^\times)$;

\item[(R3)]
$\widetilde{\mu}_1|_{F^\times} = 1$
and $\widetilde{\mu}_1 \neq 1$.
\end{enumerate}
In all cases,
the length of 
$\pi$ is two.
We denote by $\pi_1$ the unique irreducible generic 
subquotient of $\pi$
and by $\pi_2$ the remaining one.
By Theorem~\ref{thm:new} (i),
$\pi_1$ admits a newform.
For $i = 1,2$, we write $V_i$ for the space of $\pi_i$.
By \cite{BZ} (2.4),
we obtain
\begin{eqnarray}\label{eq:red_decomp}
\dim V(n) = \dim V_1(n) + \dim V_2(n),\ n \geq 0.
\end{eqnarray}
So we have
$N_{\pi_1} \geq N_\pi$.
If $\pi_2$ also admits a newform,
then
we get $N_{\pi_2} \geq N_\pi$.
The representations
$\pi_1$ and $\pi_2$ have the same central character $\omega_\pi$,
and hence we have $n_\pi = n_{\pi_1}= n_{\pi_2}$.

\subsection{Ramified case}
We shall consider the case when $\mu_1$ is ramified.
By Theorem~\ref{thm:cond} (i)
and Proposition~\ref{prop:cent},
we have $N_\pi = 2c(\mu_1)+c(\mu_2) \geq 2$ and $N_\pi > n_\pi$.

\begin{prop}\label{prop:ram}
Suppose that $\mu_1$ is ramified.
Then 

(i) $\dim V_1(n) = \dim V(n)$  for all $n \geq 0$.
In particular, $N_{\pi_1} = N_\pi$;

(ii) $V_2(n) = \{0\}$ for all $n \geq 0$.
\end{prop}
\begin{proof}
Suppose that $\pi_2$ admits a newform.
Then we have $N_{\pi_2} \geq N_\pi \geq 2$
and $N_{\pi_2} \geq N_{\pi} > n_{\pi} = n_{\pi_2}$.
This contradicts \cite{M} Theorem 4.4.
So  we conclude that $\pi_2$  has no $K_n$-fixed vectors.
This implies (ii).
The assertion (i) follows from (ii) and
(\ref{eq:red_decomp}).
\end{proof}

\subsection{Unramified case}
The case when $\mu_1$ is unramified
is slightly complicated.
In this case, we have $N_\pi = n_\pi = c(\mu_2)$
by 
Proposition~\ref{prop:cent}.
Note that $\widetilde{\mu}_1$ agrees with $\mu_1$ on $F^\times$.
It is easy to observe that 
the conditions (R1)-(R3)
are equivalent to the followings 
when $\mu_1$ is unramified:
\begin{enumerate}
\item[(RU1)]
$\mu_1 = |\cdot|_E^{\pm}$ and
$\mu_2$ is trivial;

\item[(RU2)]
${\mu}_1|_{F^\times} = \omega_{E/F} |\cdot|_F^{\pm}$;

\item[(RU3)]
$\mu_1$ is trivial 
and $\mu_2$ is not trivial.
\end{enumerate}

The injectivity of $\theta'$
is a main tool in our investigation.
We determine the representations 
$\mathrm{Ind}_B^G \mu_1 \otimes \mu_2$
for which $\theta'$ is injective on $V(N_\pi)$.
\begin{lem}\label{lem:theta_inj}
Let $\pi = \mathrm{Ind}_B^G \mu_1 \otimes \mu_2$ be a parabolically induced representation,
where $\mu_1$ is an unramified quasi-character of $E^\times$.
Then the level raising operator $\theta':  V(N_\pi) \rightarrow
V(N_\pi+1)$ is injective
if and only if $\mu_1|_{F^\times} \neq \omega_{E/F} |\cdot|_F^{-1}$.
\end{lem}
\begin{proof}
By Theorem~\ref{thm:cond},
we see that $N_\pi = c(\mu_2)$
and the function $f_{N_\pi, N_\pi}$ forms 
a basis for the one-dimensional space $V(N_\pi)$.
Thus,
every function in
$V(N_\pi)$ is
supported on $BK_{N_\pi}$.
This implies that
an element $f$ in $V(N_\pi)$
is not zero if and only if $f(1) \neq 0$.
Similarly, all elements in 
$V(N_\pi+1)$ are
supported on $BK_{N_\pi+1}$.
Hence $\theta'f \neq 0$ if and only if $\theta' f(1) \neq 0$.
By (\ref{eq:theta}), we have
\begin{eqnarray*}
\theta' f(1) & = & 
f(\zeta^{-1}) + \sum_{x \in \mi_F^{-1-n}/\mi^{-n}_F}f(u(0, x\e))\\
& = & 
|\p^{-1}|_E\mu_1(\p^{-1})f(1) + qf(1)\\
& = & 
q(q\mu_1(\p^{-1})+1)f(1).
\end{eqnarray*}
So $\theta'$ is injective if and only if
$\mu_1(\p) \neq -q$.
Because we are assuming that $\mu_1$ is unramified,
the condition $\mu_1(\p) \neq -q$ is equivalent to 
$\mu_1|_{F^\times} \neq \omega_{E/F} |\cdot|_F^{-1}$,
so this completes the proof.
\end{proof}

Firstly,
we consider the case (RU1).
Suppose that
$\mu_1 = |\cdot|_E^{\pm}$ and
$\mu_2$ is trivial.
Then we have $N_\pi = 0$
by Theorem~\ref{thm:cond} (i).
It is well-known that $\pi_1$ is the Steinberg representation of $G$
and $\pi_2$ is the trivial representation of $G$.
\begin{prop}\label{prop:St}
Suppose that $\mu_1$ and $\mu_2$ satisfy the condition (RU1).
Then

(i) $N_{\pi_1} = 2$
and
\[
\dim V_1(n) = \displaystyle \left\lfloor \frac{n-N_{\pi_1}}{2}\right\rfloor+1,\ n \geq N_{\pi_1};
\]

(ii)
$N_{\pi_2} = 0$ and
\[
\dim V_2(n) = 1\ for\ all\ n \geq 0.
\]
\end{prop}
\begin{proof}
Since $\pi_2$ is the trivial representation of $G$,
we have $\dim V_2(n) = 1$ for all $n \geq 0$.
This implies (ii).
For $n \geq 0$,
we get $\dim V_1(n) = \dim V(n)-1
= \left\lfloor \frac{n-2}{2}\right\rfloor+1$
by (\ref{eq:red_decomp})
and
Theorem~\ref{thm:cond} (ii).
In particular, we obtain $N_{\pi_1} = 2$.
This completes the proof of (i).
\end{proof}

Secondly,
we shall consider the case (RU2).
For the moment,
we fix our notation as follows.
Let $\mu_1$ be the unramified quasi-character of $E^\times$
such that
${\mu}_1|_{F^\times} = \omega_{E/F} |\cdot|_F^{-1}$
and let $\mu_2$ be a character of $E^1$.
We set $\pi = \mathrm{Ind}_B^G \mu$,
where $\mu = \mu_1 \otimes \mu_2$.
By  Lemma~\ref{lem:theta_inj},
the level raising operator $\theta'$ is zero on $V(N_\pi)$
since $\dim V(N_\pi) = 1$.
It is known that 
$(\pi, V)$ has the unique subrepresentation, which we denote 
by $(\tau, U)$.
The quotient representation $\rho = \pi/\tau$
on $W = V/U$
is the unique subrepresentation of $\pi^w = \mathrm{Ind}_B^G
\mu^w$,
where 
$\mu^w = {\overline{\mu_1}}^{-1} \otimes \mu_2$.
We write the space of $\pi^w$ as $V^w$.
Due to Theorem~\ref{thm:cond} (i),
we have
$N_\pi = N_{\pi^w} = c(\mu_2)$.
\begin{lem}\label{lem:ru2_1}
With the notation as above,
we have
\[
\dim U(N_\pi) = \dim W(N_\pi+1) = 1,\
\dim U(N_\pi+1) = \dim W(N_\pi) = 0.
\]
In particular, both
$\tau$ and $\rho$ admit newforms
and those conductors are given by
$N_\tau = N_\pi$,
$N_\rho = N_\pi+1$.
\end{lem}
\begin{proof}
By Theorem~\ref{thm:cond} (ii),
we have
$\dim V(N_\pi) = \dim V(N_\pi+1) =1$.
By \cite{BZ} (2.4),
for $n \geq 0$,
we get an exact sequence
\[
0 \rightarrow U(n) \rightarrow V(n) \rightarrow W(n) 
\rightarrow 0.
\]
Under the identification $W = V/U$,
the factor map $\theta': V(N_{\pi})/U(N_{\pi}) \rightarrow
V(N_{\pi}+1)/U(N_{\pi}+1)$
coincides with 
$\theta': W(N_{\pi}) \rightarrow
W(N_{\pi}+1)$.
So $\theta'$ is zero on $W(N_{\pi})$
because $\theta'$ on $V(N_\pi)$ is zero.
Similarly,
the factor map $\theta': V^w(N_{\pi})/W(N_{\pi}) \rightarrow
V^w(N_{\pi}+1)/W(N_{\pi}+1)$
coincides with 
$\theta': U(N_{\pi}) \rightarrow
U(N_{\pi}+1)$.
By 
Lemma~\ref{lem:theta_inj},
the map $\theta': V^w(N_{\pi}) \rightarrow
V^w(N_{\pi}+1)$
is bijective
since $\dim V^w(N_\pi) = \dim V^w(N_\pi+1) =1$.
So the restriction of $\theta'$ to $W(N_{\pi})$ 
is injective
and
the factor map $\theta': U(N_{\pi}) \rightarrow
U(N_{\pi}+1)$
is surjective.
Because $\theta'$ on $W(N_{\pi})$ 
is injective and zero,
we have $\dim W(N_{\pi}) = 0$,
and hence $\dim U(N_\pi) = \dim V(N_\pi)
-\dim W(N_\pi) = 1$.
Comparing dimensions,
we obtain $U(N_\pi) = V(N_\pi)$.
Since $\theta': U(N_\pi) \rightarrow U(N_\pi+1)$
is surjective and zero,
we have $\dim U(N_\pi+1) = 0$,
so that $\dim W(N_\pi+1) 
= \dim V(N_\pi+1)
-\dim U(N_\pi+1) = 1$.
This completes the proof.
\end{proof}

Although the following lemma is not new,
we give a proof for the reader's convenience.
\begin{lem}\label{lem:ru2_2}
With the notation as above,
$\rho$ is generic and $\tau$ is non-generic.
\end{lem}
\begin{proof}
Theorem~\ref{thm:cond} (ii) says that
the space
$V^w(N_\pi+1)$ is one-dimensional.
By
Lemma~\ref{lem:ru2_1},
we have
$W(N_\rho) = V^w(N_\pi+1)$.
It follows from Theorem~\ref{thm:cond} (ii) 
that
every function in
$V^w(N_\pi+1)$ is 
supported on
$BK_{N_\pi+1}$.
Thus, 
for any non-zero element $f$ in $V^w(N_\pi+1)$,
we have $f(1) \neq 0$.
Applying the argument in the proof of Lemma~\ref{lem:theta_inj},
we get
$\theta' f(1) \neq 0$.
This implies that
$\theta'$ on $W(N_\rho)$
is injective.
So Lemma~\ref{lem:theta} (i) says that
$\rho$ must be generic.
Since $\pi$ has the unique irreducible generic subquotient, 
the remaining representation $\tau$ is not generic.
\end{proof}

Now we get the dimension formula
of oldforms for representations in the
 case (RU2).
\begin{prop}\label{prop:ru2}
Suppose that $\mu_1$ satisfies the condition (RU2).
Then

(i) $N_{\pi_1} = N_\pi+1$ and 
\[
\dim V_1(n) = \left\lfloor \frac{n-N_{\pi_1}}{2}\right\rfloor+1,\
n \geq N_{\pi_1};
\]

(ii) $N_{\pi_2} = N_\pi$
and
\[
\dim V_2(n) = \frac{1+(-1)^{n-N_{\pi_2}}}{2},\
n \geq N_{\pi_2}.
\]
\end{prop}
\begin{proof}
We may assume that 
$\mu_1$ is the unramified quasi-character of $E^\times$
such that
${\mu}_1|_{F^\times} = \omega_{E/F} |\cdot|_F^{-1}$
since $\pi$ and $\pi^w$ have
the same irreducible subquotients.
Then Lemma~\ref{lem:ru2_2}
implies that
$(\pi_1, V_1) = (\rho, W)$
and $(\pi_2, V_2) = (\tau, U)$.
Due to Lemma~\ref{lem:ru2_1},
we have
$N_{\pi_1} = N_\pi+1$ and $N_{\pi_2} = N_\pi$.

As seen
in the proof of  Lemma~\ref{lem:ru2_2},
the operator
$\theta'$ is injective on $V_1(N_{\pi_1})$.
So by Lemma~\ref{lem:theta} (ii)
and Proposition~\ref{prop:oldform},
we obtain
$\dim V_1(n) = \left\lfloor \frac{n-N_{\pi_1}}{2}\right\rfloor+1$,
for
$n \geq N_{\pi_1}$.
The dimension formula of $V_2(n)$
follows from Theorem~\ref{thm:cond} (ii) and (\ref{eq:red_decomp}).
\end{proof}

Thirdly,
we consider the case (RU3).
Suppose that 
$\mu_1$ is trivial and 
and $\mu_2$ is not trivial.
Then we have $N_{\pi} = c(\mu_2) \geq 1$
by Theorem~\ref{thm:cond} (i).
In this case,
both $\pi_1$ and $\pi_2$ are subrepresentation of $\pi$
since $\mu$ is stable under the action of the long element of 
the affine Weyl group of $G$.
\begin{prop}\label{prop:ru3}
Suppose that $\mu_1$ and $\mu_2$ satisfy the condition (RU3).
Then

(i)
$N_{\pi_1} = N_\pi$ and 
\[
\dim V_1(n) = \left\lfloor \frac{n-N_{\pi_1}}{2}\right\rfloor +
1,\  n \geq N_{\pi_1};
\]

(ii)
$\dim V_2(n) = 0$ for all $n \geq 0$.
\end{prop}
\begin{proof}
There exists unique $i \in \{1,2\}$ such that $N_{\pi_i} = N_\pi$
since $V(N_\pi)$ is one-dimensional.
By Lemma~\ref{lem:theta_inj},
the operator $\theta'$  is injective on $V_i(N_{\pi_i}) = V(N_\pi)$.
Thus
Lemma~\ref{lem:theta} (i) implies that
$\pi_i$ must be generic.
Hence we get $i = 1$.
Due to Lemma~\ref{lem:theta} (ii) and Proposition~\ref{prop:oldform},
we obtain
$\dim V_1(n) = \left\lfloor \frac{n-N_{\pi_1}}{2}\right\rfloor +
1$.
Now Theorem~\ref{thm:cond} (ii)
implies that $V_1(n) = V(n)$ for all $n \geq 0$,
and hence we get $V_2(n) = \{0\}$ for all $n \geq 0$
by (\ref{eq:red_decomp}).
\end{proof}

\Section{Explicit newforms}\label{sec:enew}
In this section,
we determine newforms for generic non-supercuspidal
representations explicitly.
Every irreducible 
non-supercuspidal representation $\pi$
of $G$ can be embedded into
$\mathrm{Ind}_B^G \mu$,
for some quasi-character $\mu$ of $T$.
We shall realize the newforms for $\pi$
as functions 
in $\mathrm{Ind}_B^G \mu$.
This problem is easy 
unless $\pi$ is the Steinberg representation of $G$:
\begin{prop}\label{prop:enew}
Let $(\pi, V)$ be an irreducible generic non-supercuspidal
representation of $G$,
which is not isomorphic to the Steinberg representation of $G$.
Let $\mu_1$ be a quasi-character of $E^\times$
and $\mu_2$ a character of $E^1$
such that $\pi$ is a subrepresentation of 
$\mathrm{Ind}_B^G \mu_1\otimes \mu_2$.
Then the function
$f_{N_\pi, N_\pi-c(\mu_1)}$ is a newform for $\pi$.
Here $f_{N_\pi, N_\pi-c(\mu_1)}$
is the function in Definition~\ref{defn:f}.
\end{prop}
\begin{proof}
By Theorem~\ref{thm:cond} (i),
$\mathrm{Ind}_B^G \mu_1\otimes \mu_2$
has conductor $2c(\mu_1) + c(\mu_2)$.
If $\pi = \mathrm{Ind}_B^G \mu_1\otimes \mu_2$,
then the assertion follows from Theorem~\ref{thm:cond} (ii).
Suppose that $\pi$ is a proper submodule of 
$\mathrm{Ind}_B^G \mu_1\otimes \mu_2$.
Since
$\pi$ is not the representation in the case (RU1),
it follows from Propositions~\ref{prop:ram},
\ref{prop:ru2} and \ref{prop:ru3}
that
$N_\pi$ equals to 
$2c(\mu_1) + c(\mu_2)$ or $2c(\mu_1) + c(\mu_2)+1$.
Theorem~\ref{thm:cond} (ii)
says that
the space of $K_{N_\pi}$-fixed vectors in 
$\mathrm{Ind}_B^G \mu_1\otimes \mu_2$
is one-dimensional.
So $V(N_\pi)$ is just the space of 
$K_{N_\pi}$-fixed functions in 
$\mathrm{Ind}_B^G \mu_1\otimes \mu_2$.
Thus the proposition follows from
Theorem~\ref{thm:cond} (ii).
\end{proof}

We shall
determine newforms for the Steinberg representation
$\mathrm{St}_G$ of $G$.
Let $\pi_1$ be the Steinberg representation of $G$.
Then
$\pi_1$
is the unique subrepresentation
of $\pi = \mathrm{Ind}_B^G (\mu_1 \otimes \mu_2)$,
where 
$\mu_1 = |\cdot|_E$ and $\mu_2 = 1$.
We have
$N_{\pi_1} = 2$, $N_\pi = 0$ and $n_{\pi_1} = n_\pi = 0$ by Propositions~\ref{prop:cent}, \ref{prop:St} (i)
and Theorem~\ref{thm:cond} (i).
We write $V$ and $V_1$ for the spaces of 
$\pi$ and $\pi_1$ respectively.
Due to Theorem~\ref{thm:cond} (ii),
$V(2)$ is the two-dimensional
subspace of $V$
spanned by $f_{2,1}$ and $f_{2,2}$.
The space $V_1(2)$ of newforms for $\pi_1$
is a one-dimensional subspace of $V(2)$.
\begin{prop}\label{prop:St2}
With the notation as above,
a function $f$ in $V(2)$ lies in 
$V_1(2)$
if and only if
$f(1) = -q(q-1)f(\gamma_1)$.
In particular,
$q(q-1)f_{2,2}-f_{2,1}$
is a newform for the Steinberg representation of $G$.
\end{prop}
\begin{proof}
Since $\pi$ has trivial central character,
the group $Z_1K_2$ acts on $V(2)$
trivially.
We define a level lowering operator $\delta: V(2) \rightarrow V(1)$ by
\[
\delta v = 
\frac{1}{\mathrm{vol}(K_{1}\cap (Z_{1}K_2))}
\int_{K_{1}}\pi(k)v dv,\ v \in V(2).
\]
The space $V_1(2)$ is contained in $\ker \delta$
since $V_1(1) = \{0\}$.
We shall show that 
a function $f$ in $V(2)$ lies in 
the kernel of $\delta$
if and only if
$f(1) = -q(q-1)f(\gamma_1)$.
Then $\ker \delta$ is of dimension one,
and hence coincides with $V_1(2)$.

It follows from \cite{M2} Lemma 4.9 that
$\delta$ has the following form:
\[
\delta v = \sum_{\substack{y \in \mi_E/\mi_E^{2}\\z 
\in \mi_F/\mi_F^{2}}}
\pi(\hat{u}(y, z\e-y\overline{y}/2)) v+ 
 \sum_{y \in \mi_E^{-1}/\ri_E}
\pi( \zeta {u}(y, -y\overline{y}/2)) v,\ v \in V(2).
\]
For any element $f$ in $V(2)$,
we have
\begin{eqnarray*}
\delta f(1) & = & 
 \sum_{\substack{y \in \mi_E/\mi_E^{2}\\z 
\in \mi_F/\mi_F^{2}}}
\pi(\hat{u}(y, z\e-y\overline{y}/2)) f(1)+
 \sum_{y \in \mi_E^{-1}/\ri_E}
\pi( \zeta {u}(y, -y\overline{y}/2)) f(1)\\
& = & 
 \sum_{\substack{y \in \mi_E/\mi_E^{2}\\z 
\in \mi_F/\mi_F^{2}}}
f(\hat{u}(y, z\e-y\overline{y}/2)) +
 \sum_{y \in \mi_E^{-1}/\ri_E}
f( \zeta {u}(y, -y\overline{y}/2))\\
& = & 
  \sum_{\substack{y \in \mi_E/\mi_E^{2}\\z 
\in \mi_F/\mi_F^{2}}}
f(\hat{u}(y, z\e-y\overline{y}/2)) +
q^{-2}\mu_1(\p) \sum_{y \in \mi_E^{-1}/\ri_E}
f({u}(y, -y\overline{y}/2))\\
& = & 
 \sum_{\substack{y \in \mi_E/\mi_E^{2}\\z 
\in \mi_F/\mi_F^{2}}}
f(\hat{u}(y, z\e-y\overline{y}/2)) +
\mu_1(\p) 
f(1).
\end{eqnarray*}
We shall compute
$f(\hat{u}(y, z\e-y\overline{y}/2)) $, for $y \in \mi_E/\mi_E^{2}$ and $z \in \mi_F^{}/\mi_F^{2}$.

(i)
Suppose that 
$y \in \mi_E^{2}$ and $z \in \mi_E^{2}$.
Then we have
$f(\hat{u}(y, z\e-y\overline{y}/2)) =f(1)$
since $\hat{u}(y, z\e-y\overline{y}/2)$ lies in $K_2$.

(ii) If 
$y \not\in \mi_E^{2}$ and $z \in \mi_E^{2}$,
then $\hat{u}(y, z\e-y\overline{y}/2)=
\hat{u}(y, -y\overline{y}/2)\hat{u}(0, z\e) \equiv \hat{u}(y, -y\overline{y}/2) \pmod{K_{2}}$.
There exists $a \in \ri_E^\times$ such that
$t(a)\hat{u}(y, -y\overline{y}/2)t(a)^{-1}
= \hat{u}(\p, -\p^2/2) = \gamma_1$.
So we have
$f(\hat{u}(y, z\e-y\overline{y}/2)) =f(\gamma_1)$.

(iii) Suppose that 
$z \not\in \mi_E^{2}$.
Then $x = z\e-y\overline{y}/2$ 
lies in $\mi_E\backslash \mi_E^2$.
We have
\[
\hat{u}(y, z\e-y\overline{y}/2) =
 u(-\overline{y}/\overline{x}, 1/x) \mathrm{diag}(\p^{2}/\overline{x}, -\overline{x}/x, \p^{-2}x)t_{2}u(-\overline{y}/x, 1/\overline{x}).
\]
Since $t_{2}u(-\overline{y}/x, 1/\overline{x})\in K_{2}$
and $\p^{2}/\overline{x} \in \p \ri_E^\times$,
we obtain
\[
f(\hat{u}(y, z\e-y\overline{y}/2)) = q^{-2}\mu_1(\p)\mu_2(-\overline{x}/x)f(1)
=
q^{-2}\mu_1(\p)f(1).
\]

Therefore we get 
\begin{eqnarray*}
\delta f(1) & = & 
 \sum_{\substack{y \in \mi_E/\mi_E^{2}\\z \in \mi_F/\mi_F^{2}}}
f(\hat{u}(y, z\e-y\overline{y}/2)) +
\mu_1(\p) 
f(1)\\
& = & 
f(1)+(q^2-1)f(\gamma_1)+q^2(q-1)q^{-2}\mu_1(\p)f(1)
+\mu_1(\p) 
f(1)\\
& = & 
(q\mu_1(\p)+1)f(1)+(q^2-1)f(\gamma_1)\\
& = & 
(q^{-1}+1)f(1)+(q^2-1)f(\gamma_1).
\end{eqnarray*}
By Lemma~\ref{lem:coset},
we have $G = BK_1$.
Therefore
$\delta f \in V(1)$ is zero 
if and 
only if $\delta f(1) = 0$.
So we conclude that $\delta f = 0$
if and only if
$f(1) = -q(q-1)f(\gamma_1)$.
This completes the proof.
\end{proof}

\begin{cor}\label{cor:1}
Suppose that  an irreducible generic 
representation $(\pi, V)$ of $G$ is a subrepresentation of 
$\mathrm{Ind}_B^G \mu_1 \otimes \mu_2$,
where $\mu_1$ is an unramified quasi-character of $E^\times$.
We regard elements in $V$
as functions in $\mathrm{Ind}_B^G \mu_1 \otimes \mu_2$.
Then
any non-zero element in $f \in V(N_\pi)$
satisfies $f(1) \neq 0$.
\end{cor}
\begin{proof}
If $\pi$ is not isomorphic to the Steinberg representation
of $G$,
then it follows from Proposition~\ref{prop:enew}
 that a non-zero element in $f \in V(N_\pi)$
 is a non-zero multiple of $f_{N_\pi, N_\pi}$.
The assertion follows because 
the support of $f_{N_\pi, N_\pi}$ is $BK_{N_\pi}$.
Suppose that $\pi$ is the Steinberg representation
of $G$.
Then we have $\mu_1 = |\cdot|_E$ and $\mu_2 = 1$.
By Proposition~\ref{prop:St2},
a non-zero element in $f \in V(N_\pi)$
 is a non-zero multiple of $q(q-1)f_{2,2}-f_{2,1}$.
 Since $(q(q-1)f_{2,2}-f_{2,1})(1) = q(q-1)$,
the proof is complete.
\end{proof}

\Section{Test vectors for the Whittaker functional}\label{sec:old}
We close this paper by showing that
newforms for generic representations of $G$
are test vectors for the Whittaker functional.
\begin{prop}\label{prop:test}
Let $(\pi, V)$ be an irreducible generic representation of $G$.
Then we have $W_v(1) \neq 0$ for all
non-zero $v \in V(N_\pi)$.
\end{prop}
\begin{proof}
The assertion follows from \cite{M} Theorem 4.12
if $\pi$ satisfies
$N_\pi \geq 2$ and $N_\pi > n_\pi$.
By \cite{M} Corollary 5.5,
all generic supercuspidal representations satisfy this condition.
So we may assume that $\pi$ is not supercuspidal.
As usual,
we embed $\pi$ into $\tau = \mathrm{Ind}_B^G \mu_1 \otimes \mu_2$.
If $\mu_1$ is ramified,
we have
$N_\pi \geq N_\tau = 2c(\mu_1) + c(\mu_2)\geq 2$
and $N_\pi \geq N_\tau > n_\tau = n_\pi$
by Theorem~\ref{thm:cond} (i) and Proposition~\ref{prop:cent}.
So we may suppose $\mu_1$ is unramified.

(I) Suppose that $N_\pi = 0$.
We claim that $\tau$ is irreducible.
Then $\pi = \mathrm{Ind}_B^G \mu_1 \otimes \mu_2$
is an unramified principal series representation
by Theorem~\ref{thm:cond} (i).
Hence
the assertion follows from \cite{CS} Theorem 5.4.
We shall show the claim.
We assume that $\tau$ is reducible.
If $\mu_1$ and $\mu_2$ satisfy the conditions (RU1)
or (RU2),
then 
by Propositions~\ref{prop:St} and \ref{prop:ru2},
$N_\pi$ must be positive.
Suppose that $\mu_1$ and $\mu_2$ satisfy the condition (RU3).
Then 
we have $N_\pi \geq N_\tau = c(\mu_2) > 0$
by Theorem~\ref{thm:cond} (i).
This completes the proof of the claim.

(II) Suppose that $N_\pi \geq 1$.
Then, by Lemma~\ref{lem:theta},
it is enough to show that 
$\theta'$ is injective on $V(N_\pi)$.
We regard any element in $V$ as a function in 
$\mathrm{Ind}_B^G \mu_1 \otimes \mu_2$.
By Corollary~\ref{cor:1},
any non-zero element in $V(N_\pi)$
satisfies $f(1) \neq 0$.
Since $\pi$ is the generic subrepresentation
of $\mathrm{Ind}_B^G \mu_1 \otimes \mu_2$,
Lemma~\ref{lem:ru2_2} implies 
$\mu_1|_{F^\times} \neq \omega_{E/F} |\cdot|_F^{-1}$.
Applying the argument in the proof of Lemma~\ref{lem:theta_inj},
we obtain $\theta' f(1) \neq 0$.
This implies that $\theta'$ is injective.
The proof is now complete.
\end{proof}

By Proposition~\ref{prop:oldform},
we obtain the following
\begin{cor}\label{cor:oldform}
Let $(\pi, V)$ be an irreducible generic representation of $G$.
Then, for $n \geq N_\pi$,
the set
$\{\theta^{'i} \eta^j v\ |\ i+2j+N_\pi = n\}$
forms a basis for $V(n)$.
In particular,
\[
\dim V(n) = \left\lfloor \frac{n-N_\pi}{2}\right\rfloor +1.
\]
\end{cor}


\begin{thebibliography}{10}

\bibitem{Baruch}
E.~M. Baruch.
\newblock On the gamma factors attached to representations of {${\rm U}(2,1)$}
  over a {$p$}-adic field.
\newblock {\em Israel J. Math.}, 102:317--345, 1997.

\bibitem{BZ}
I.~N. Bern{\v{s}}te{\u\i}n and A.~V. Zelevinski{\u\i}.
\newblock Representations of the group {$GL(n,F),$} where {$F$} is a local
  non-{A}rchimedean field.
\newblock {\em Uspehi Mat. Nauk}, 31(3(189)):5--70, 1976.

\bibitem{Casselman}
W.~Casselman.
\newblock On some results of {A}tkin and {L}ehner.
\newblock {\em Math. Ann.}, 201:301--314, 1973.

\bibitem{CS}
W.~Casselman and J.~Shalika.
\newblock The unramified principal series of {$p$}-adic groups. {II}. {T}he
  {W}hittaker function.
\newblock {\em Compositio Math.}, 41(2):207--231, 1980.

\bibitem{GPS}
S.~Gelbart and I.~Piatetski-Shapiro.
\newblock Automorphic forms and {$L$}-functions for the unitary group.
\newblock In {\em Lie group representations, {II} ({C}ollege {P}ark, {M}d.,
  1982/1983)}, volume 1041 of {\em Lecture Notes in Math.}, pages 141--184.
  Springer, Berlin, 1984.

\bibitem{JPSS}
H.~Jacquet, I.~Piatetski-Shapiro, and J.~Shalika.
\newblock Conducteur des repr\'esentations du groupe lin\'eaire.
\newblock {\em Math. Ann.}, 256(2):199--214, 1981.

\bibitem{Keys}
D.~Keys.
\newblock Principal series representations of special unitary groups over local
  fields.
\newblock {\em Compositio Math.}, 51(1):115--130, 1984.

\bibitem{LR}
J.~Lansky and A.~Raghuram.
\newblock Conductors and newforms for {$U(1,1)$}.
\newblock {\em Proc. Indian Acad. Sci. (Math Sci.)}, 114(4):319--343, 2004.

\bibitem{LRS}
J.~Lansky and A.~Raghuram.
\newblock Conductors and newforms for {$\rm SL(2)$}.
\newblock {\em Pacific J. Math.}, 231(1):127--153, 2007.

\bibitem{M2}
M.~Miyauchi.
\newblock On epsilon factors of supercuspidal representations of unramified
  {$\mathrm{U}(2,1)$}, preprint, arxiv:1111.2212, 2011.

\bibitem{M}
M.~Miyauchi.
\newblock On local newforms for unramified {$\mathrm{U}(2,1)$}, preprint,
  ar{X}iv:1105.6004, 2011.

\bibitem{Morris-2}
L.~Morris.
\newblock Tamely ramified supercuspidal representations of classical groups.
  {I}. {F}iltrations.
\newblock {\em Ann. Sci. \'Ecole Norm. Sup. (4)}, 24:705--738, 1991.

\bibitem{Reeder}
M.~Reeder.
\newblock Old forms on {${\rm GL}_n$}.
\newblock {\em Amer. J. Math.}, 113(5):911--930, 1991.

\bibitem{RS}
B.~Roberts and R.~Schmidt.
\newblock {\em Local newforms for {GS}p(4)}, volume 1918 of {\em Lecture Notes
  in Mathematics}.
\newblock Springer, Berlin, 2007.

\end{thebibliography}
\end{document}